\documentclass[11pt]{article}
\usepackage{mathrsfs}
\usepackage{amsmath, mathtools}
\usepackage{amssymb}
\usepackage{amsthm}
\usepackage{graphicx}
\usepackage{epic}
\usepackage{xcolor,cite}
\usepackage{cases}
\usepackage{pst-poly}
\usepackage{pst-plot}
\usepackage{CJK,makecell}

\usepackage{verbatim}
\usepackage{subfig}
\usepackage{booktabs}
\definecolor{dkgreen}{rgb}{0,0.6,0}
\numberwithin{equation}{section}

\usepackage{etex}

\usepackage[left,mathlines,displaymath]{lineno}

\usepackage{pgf}
\renewcommand{\paragraph}{\roman{paragraph}}

\usepackage{bm}
\usepackage[hidelinks]{hyperref}
\usepackage{tikz}
\usetikzlibrary{automata}
\usepackage{enumitem}
\usetikzlibrary{calc}
\usetikzlibrary{arrows,shapes,positioning}
\usetikzlibrary{decorations.markings}
\tikzstyle arrowstyle=[scale=1]
\tikzstyle directed=[postaction={decorate,decoration={markings, mark=at position .65 with {\arrow[arrowstyle]{stealth}}}}]
\tikzstyle reverse directed=[postaction={decorate,decoration={markings, mark=at position .65 with {\arrowreversed[arrowstyle]{stealth};}}}]

\newtheorem{theorem}{Theorem}[section]
\newtheorem{corollary}[theorem]{Corollary}

\newtheorem{conjecture}[theorem]{Conjecture}

\newtheorem{lemma}[theorem]{Lemma}

\newcommand{\JCTB}{{\it J. Combin. Theory Ser. B.} }

\newcommand{\st}{\operatorname{st}}
\newcommand{\tr}{\operatorname{tr}}
\newcommand{\ex}{\operatorname{ex}}

\def \Forb {\mathrm{Forb}}

\usepackage{color}

\usepackage{authblk}%%%for authors

\begin{document}

\title{A Tur\'an-type extremal problem for the number of spanning trees in $C_4$-free graphs}

\author[1,2\thanks{Corresponding author. Email: xush0928@163.com}]
{Shaohan Xu}

\author[3\thanks{Email: fengming.dong@nie.edu.sg and donggraph@163.com.}]
{Fengming Dong}

\author[1,2\thanks{Email: kexxu1221@126.com}]
{Kexiang Xu}

\affil[1]{\footnotesize School of Mathematics,
	Nanjing University of Aeronautics and Astronautics,
	Nanjing, Jiangsu, 210016, PR China}

\affil[2]
{\footnotesize
MIIT Key Laboratory of Mathematical Modelling and High Performance Computing of Air  Vehicles, Nanjing, Jiangsu, 210016, PR China
}

\affil[3]{\footnotesize National Institute of Education, Nanyang Technological University, Singapore}

\date{}

\maketitle

\begin{abstract}
For a graph \(F\), the Tur\'an number \(\ex(n,F)\) is the maximum number of edges in an \(F\)-free graph on \(n\) vertices. Let \(q\ge 2\) be an integer and set \(n=q^{2}+q+1\). Brown and Erd\H{o}s, R\'enyi and S\'os independently proved that $\ex(n,C_{4})\ge \frac12 q(q+1)^{2}$ for every prime power \(q\), and F\"uredi subsequently established the upper bound $\frac12 q(q+1)^{2}$ for $\ex(n,C_{4})$ whenever \(q\notin\{1, 7,9,11,13\}\). In this article, we prove that every \(C_{4}\)-free graph \(G\) on \(n\) vertices with at most \(\frac12 q(q+1)^{2}\) edges satisfies $\tau(G)\le n^{(n-3)/2}$, where \(\tau(G)\) denotes the number of spanning trees of \(G\). In particular, for every prime power $q\notin\{7,9,11,13\}$, the above upper bound on $\tau(G)$ is attained precisely by the orthogonal polarity graphs, thereby proving London's conjecture for all such $q$.
\end{abstract}

\noindent{\bf Keywords:} spanning trees, $C_4$-free graphs, extremal graph theory, polarity graphs, Tur\'an-type problems\\

\noindent{{\bf 2020 Mathematics Subject Classification:} 05C30, 05C35, 05C50, 51E20}

%\tableofcontents\Large

\section{Introduction}
Throughout the paper, all graphs are finite and simple unless stated otherwise. Let $G$ be a graph with vertex set $V(G)$ and edge set $E(G)$. We write $e(G)=|E(G)|$ for the number of edges in $G$. For a vertex $v\in V(G)$, let $N_G(v)$ denote the {\it neighborhood} of $v$ in $G$, and let $d_G(v)=|N_G(v)|$ denote its {\it degree}. When the underlying graph is clear from the context, we simply write $N(v)$ and $d_v$ for $N_G(v)$ and $d_G(v)$, respectively.

A graph is called $F$-{\it free} if it contains no subgraph isomorphic to the graph $F$. For any positive integer $n$, let $\Forb(n, F)$ denote the set of $F$-free graphs of order $n$. The {\it Tur\'an number} of $F$, denoted by $\ex(n,F)$, is the maximum number of edges among all graphs in $\Forb(n,F)$. One of the fundamental examples in extremal graph theory is Mantel's theorem \cite{Mantel1907}, which states that
\[\ex(n,C_3)=\left\lfloor\frac{n^2}{4}\right\rfloor,\]
with equality attained by the balanced complete bipartite graph $K_{\lfloor n/2\rfloor,\lceil n/2\rceil}$. Interestingly, the same graph also maximizes the number of spanning trees among all $n$-vertex $C_3$-free graphs. More generally, for integers $n$ and $r$ with $2\le r\le n$, combining Tur\'an's theorem with the spanning tree extremal result of Petingi and Rodriguez \cite{PetingiRodriguez2002}, one obtains that, among all $n$-vertex $K_{r+1}$-free graphs, the Tur\'an graph $T_r(n)$, namely the balanced complete $r$-partite graph, maximizes the number of spanning trees.

The corresponding Tur\'an problem for $C_4$-free graphs has a long history. The study of $\ex(n,C_4)$ goes back to Erd\H{o}s \cite{Erdos1938}.  Reiman \cite{Reiman1958} obtained a general upper bound $\ex(n,C_4)\le \frac{n}{4}\left(1+\sqrt{4n-3}\right)$. On the other hand, using polarities of finite projective planes, Brown
\cite{Brown1966} and Erd\H{o}s, R\'enyi and S\'os \cite{ErdosRenyiSos1966} independently obtained
\[
\ex(q^2+q+1,C_4)\ge \frac12 q(q+1)^2
\]
for every prime power $q$. Erd\H{o}s conjectured that this lower bound is best possible for all prime powers $q$. F\"uredi \cite{Furedi1983,Furedi1996} subsequently established the  corresponding upper bound for all positive integers $q$ with $q\notin\{1,7,9,11,13\}$; see also \cite[Theorem~1.1]{HMY2021}.

\begin{theorem}[\hspace{1sp}{\cite{Furedi1983,Furedi1996}}]\label{th-1}
If $q$ is a positive integer with $q\notin\{1,7,9,11,13\}$ and $n=q^2+q+1$, then
\[
\ex(n,C_4)\le \frac12q(q+1)^2.
\]
\end{theorem}

Thus the exact value of $\ex(q^2+q+1,C_4)$ is known for every prime power $q\notin\{7,9,11,13\}$. F\"uredi also proved that, for all sufficiently large integers $q$, every $C_4$-free graph on $q^2+q+1$ vertices with exactly $\frac 12 q(q+1)^2$ edges is an orthogonal polarity graph (unpublished; see \cite{Furedi2007}). The related value $\ex(q^2+q,C_4)$ was also determined by Firke, Kosek, Nash and Williford \cite{Firke2013} for all $q=2^k$. Recently, He, Ma and Yang \cite{HMY2021,HMY2023} obtained stability and further exact results related to the extremal problems on $C_4$, while Ma and Yang \cite{MaYang2023} established further upper bounds for $\ex(n,C_4)$.

For any graph $G$, let $\tau(G)$ denote the number of spanning trees in $G$. We set $\tau(G) = 0$ when $G$ is disconnected; hence all upper bounds stated for connected graphs extend immediately to arbitrary graphs of the same order. Extremal problems for the number of spanning trees have also received considerable attention in optimum design and network theory \cite{Cheng1981,BoeschLiSuffel1991,GilbertMyrvold1997,Kelmans1996,KelmansChelnokov1974}.
These results motivate the study of an analogous Tur\'an-type problem in which the number of edges is replaced by the number of spanning trees. Following London \cite{London2026}, define
\[\st(n,C_4):=\max\{\tau(G):G\in \Forb(n,C_4)\}.\]

For the following terminology, see \cite[Section~2.3]{HMY2023}. Let $\Pi$ be a finite projective plane of order $q$ with point set $\mathcal P$ and line set $\mathcal L$.  A {\it polarity} $\varphi$ of $\Pi$ is a bijection $\varphi:\mathcal P\cup\mathcal L\longrightarrow\mathcal P\cup\mathcal L$ such that $\varphi(\mathcal P)=\mathcal L$, $\varphi(\mathcal L)=\mathcal P$, $\varphi^2$ is the identity function, and for every point $x\in\mathcal P$ and line $\ell\in\mathcal L$,
\[x\in\ell\quad\Longleftrightarrow\quad\varphi(\ell)\in\varphi(x).\]
A point $x\in\mathcal P$ is called {\it absolute} if $x\in\varphi(x)$. The {\it polarity graph} $G(\varphi)$ associated with $\varphi$ is the simple graph with vertex set $\mathcal P$ such that, for two distinct points $x,y\in\mathcal P$,
\[xy\in E(G(\varphi))\quad\Longleftrightarrow\quad x\in\varphi(y).\]
A polarity $\varphi$ is called {\it orthogonal} if it has exactly $q+1$ absolute points, and in this case $G(\varphi)$ is called an {\it orthogonal polarity graph}. Note that any two distinct points lie in a unique line and $|\mathcal L|=|\mathcal P|=q^2+q+1$. It follows that every polarity graph arising from a projective plane of order $q$ has $q^2+q+1$ vertices. Moreover, suppose that \(x_1x_2x_3x_4x_1\) is a \(4\)-cycle in \(G(\varphi)\). Then the distinct points \(x_2\) and \(x_4\) both lie on each of the lines \(\varphi(x_1)\) and \(\varphi(x_3)\). The uniqueness of the line through two distinct points implies that \(\varphi(x_1)=\varphi(x_3)\), contradicting the injectivity of \(\varphi\). Hence \(G(\varphi)\) is \(C_4\)-free.

For a prime power $q$, let $\mathbb F_q$ be the field with $q$ elements, and let $\mathbb F_q^3$ denote the three-dimensional vector space over $\mathbb F_q$. The Desarguesian projective plane $\mathrm{PG}(2,q)$ has the one-dimensional subspaces of $\mathbb F_q^3$ as points and the two-dimensional subspaces as lines, with incidence given by inclusion. The classical {\it Erd\H{o}s--R\'enyi orthogonal polarity graph} $ER_q$ has the points of $\mathrm{PG}(2,q)$ as its vertices. Two distinct points represented by nonzero vectors $x=(x_0,x_1,x_2)$ and $y=(y_0,y_1,y_2)$ are adjacent if and only if $x_0y_0+x_1y_1+x_2y_2=0$. This is an orthogonal polarity graph; see \cite{PengTaitTimmons2015}.

London \cite{London2026} determined the number of spanning trees of every polarity graph arising from a projective plane of order $q$. In particular, he obtained the following sharp result within the class of polarity graphs.

\begin{theorem}[\hspace{1sp}{\cite{London2026}}]
\label{thm:London-polarity}
Let $G$ be a polarity graph arising from a projective plane of order $q$, and let $n=q^2+q+1$. Then
\[\tau(G)\le n^{(n-3)/2},\]
with equality if and only if $G$ is an orthogonal polarity graph. In particular, for every prime power $q$, $\tau(ER_q)=n^{(n-3)/2}$.
\end{theorem}

For every prime power $q$ and $n=q^2+q+1$, the graph $ER_q$ is $C_4$-free with $\tau(ER_q)=n^{(n-3)/2}$. Thus,  $\st(n,C_4)\ge n^{(n-3)/2}$. London \cite{London2026} conjectured that this lower bound is best possible and that the extremal graphs are precisely the orthogonal polarity graphs.

\begin{conjecture}[\hspace{1sp}{\cite{London2026}}]
\label{conj:London}
Let $q$ be a prime power and let $n=q^2+q+1$. Then
\[\st(n,C_4)=n^{(n-3)/2},\]
and the maximizers are precisely the orthogonal polarity graphs.
\end{conjecture}

Using the Grone-Merris bound \cite{GroneMerris1988}, together with F\"uredi's extremal edge bound and a degree balancing argument, London \cite{London2026} also proved that
\[\ln\st(n,C_4)=\frac{n-3}{2}\ln n+O(\sqrt n)\]
as $q\to\infty$ through prime powers, where $n=q^2+q+1$. Moreover, combining the exact spanning tree counting formula for polarity graphs with the stability theorem of He, Ma and Yang \cite{HMY2023}, London \cite{London2026} proved that, for every fixed $c\in(0,1)$ and all sufficiently large even integers $q$, every $C_4$-free graph $G$ on $n=q^2+q+1$ vertices with
\[e(G)\ge\frac12q(q+1)^2-\frac{c}{2}q\]
satisfies
\[\tau(G)\le n^{(n-3)/2},\]
with equality if and only if $G$ is an orthogonal polarity graph. Thus London's conjecture was proved in this stability regime for all sufficiently large even prime powers $q$.

Our main result removes the restrictions that $q$ is even and sufficiently large. It is a uniform upper bound for the number of spanning trees in connected $C_4$-free graphs of order $q^2+q+1$, valid for every integer $q\ge2$. The bound depends explicitly on the number of edges and is attained precisely by polarity graphs. For an integer $q\ge2$, put
\begin{equation}
	\label{intro-cons}
	n=q^2+q+1,\qquad
	m_q=\frac12q(q+1)^2,\qquad
	A_q=\frac{1}{2\sqrt q}
	\ln\frac{q+1+\sqrt q}{q+1-\sqrt q}.
\end{equation}

\begin{theorem}
	\label{thm:uniform-bound}
	Let $q\ge2$ be an integer, and let $G$ be a connected $C_4$-free graph on $n=q^2+q+1$ vertices. Then
	\begin{equation}\label{eq:intro-uniform-bound}
		\tau(G)\le n^{(n-3)/2}
		\exp\bigl(-2A_q(m_q-e(G))\bigr).
	\end{equation}
	Equality holds if and only if $G$ is a polarity graph arising from a projective plane of order $q$.
\end{theorem}

Since $A_q>0$, Theorem~\ref{thm:uniform-bound} immediately gives the following consequence when $e(G)\le m_q$. The equality characterization follows from the fact that a polarity graph has $m_q$ edges if and only if its polarity has exactly $q+1$ absolute points; see Lemma~\ref{pola-de}.

\begin{corollary}\label{cor:edge-bound}
	Let $q\ge2$ be an integer, and put $n=q^2+q+1$. If $G$ is a connected $C_4$-free graph on $n$ vertices satisfying
	\[
	e(G)\le\frac12q(q+1)^2,
	\]
	then
	$
	\tau(G)\le n^{(n-3)/2}.
	$
	Equality holds if and only if $G$ is an orthogonal polarity graph.
\end{corollary}

Corollary~\ref{cor:edge-bound} applies independently of the exact value of $\ex(q^2+q+1,C_4)$. Combining it with F\"uredi's edge bound in Theorem~\ref{th-1} yields the following result.

\begin{corollary}
\label{cor:nonexceptional-orders}
	Let $q\ge2$ be an integer with $q\notin\{7,9,11,13\}$, and let $n=q^2+q+1$. Then every connected $C_4$-free graph $G$ on $n$ vertices satisfies
	\[
	\tau(G)\le n^{(n-3)/2}.
	\]
	Equality holds if and only if $G$ is an orthogonal polarity graph.
\end{corollary}

For a general integer $q$, the equality characterization identifies the only possible graphs attaining the displayed upper bound; it does not assert that the bound is attained. When $q$ is a prime power, the Erd\H{o}s--R\'enyi orthogonal polarity graph $ER_q$ exists and attains the bound. Moreover, disconnected graphs have no spanning trees. Hence Corollary~\ref{cor:nonexceptional-orders} proves Conjecture~\ref{conj:London} for every prime power $q\notin\{7,9,11,13\}$.

\begin{corollary}\label{cor:london-conjecture}
Let $q$ be a prime power with $q\notin\{7,9,11,13\}$, and let $n=q^2+q+1$. Then
\[
\st(n,C_4)=n^{(n-3)/2},
\]
and the maximizers are precisely the orthogonal polarity graphs.
\end{corollary}

\section{Preliminaries}
We first record a consequence of the $C_4$-free condition. Let $G$ be a $C_4$-free graph of order $n$ and $V(G)=\{1,2,\ldots,n\}$. Define  $c_{ij}=|N(i)\cap N(j)|$ for distinct vertices $i,j\in V(G)$.  Since $G$ is $C_4$-free, we have $c_{ij}\le1$. Let $t(G)$ denote  the number  of triangles in   $G$.  For $ij\in E(G)$, we have $c_{ij}=1$ precisely when $ij$ lies in a triangle. Hence
\[
\sum_{ij\in E(G)}c_{ij}=3t(G),
\]
since each triangle contributes once for each of its three edges. By counting $2$-paths according to their middle vertex and their  two endpoints, respectively,
\[
\sum_{i\in V(G)}\binom{d_i}{2}=\sum_{\{i,j\}\subseteq V(G)}c_{ij}=\sum_{ij\notin E(G)} c_{ij}+\sum_{ij\in E(G)}c_{ij}\le\binom{n}{2}-e(G)+3t(G).
\]
Equivalently,
\begin{equation}\label{eq:triangle-lower}
6t(G)\ge\sum_{i\in V(G)}d_i^2-n(n-1).
\end{equation}

We next recall some standard matrix notation. The {\it adjacency matrix} of $G$ is defined as $A(G)=(a_{ij})_{n\times n}$, where
\[a_{ij}=
\begin{cases}
1, &\mbox {\rm if $ij\in E(G)$},\\
0, &\mbox {\rm otherwise}.
\end{cases}
\]
Let $D(G)=\operatorname{diag}(d_{1},d_{2},\ldots,d_{n})$ be the {\it degree diagonal matrix} of $G$. The {\it Laplacian matrix} of $G$ is defined as $L(G)=D(G)-A(G)$. Let $I_n$, $J_n$ and $\mathrm{j}_n$ be the {\it identity matrix} of order $n$, {\it all-ones matrix} of order $n$  and {\it all-ones column vector} of dimension $n$, respectively.

For a square matrix $M$, let $\tr(M)$ be its {\it trace}, and let $M_{j}^{i}$ denote the submatrix of  $M$ obtained by deleting the $i$-th row and $j$-th column in $M$.  We shall use the following classical matrix-tree theorem.

\begin{theorem}[Matrix-Tree Theorem, {\cite{Biggs1974,BondyMurty1976}}]\label{thm:matrix-tree}
Let $G$ be a connected graph on $n$ vertices. For any $i,j\in\{1,2,\ldots,n\}$, we have
\[\tau(G)=(-1)^{i+j}\det \left(L(G)_{j}^{i}\right).\]
Moreover, if the eigenvalues of $L(G)$ are $\lambda_1, \lambda_2,\ldots,\lambda_{n-1}$ and $\lambda_n=0$, then
\[\tau(G)=\frac{1}{n}\prod_{i=1}^{n-1}\lambda_i.\]
\end{theorem}

The following standard fact about adjacency matrices will be used repeatedly; see \cite[Proposition~1.3.1]{BrouwerHaemers2012}
for example.

\begin{lemma}[\hspace{1sp}{\cite{BrouwerHaemers2012}}]\label{lem:adjacency-walks}
Let $G$ be a  graph and $k$ be any positive integer. Then the $(i,j)$-entry of $A(G)^k$ is the number of walks of length $k$ from $i$ to $j$. In particular, $(A(G)^2)_{ii}=d_i$, and
\[\tr(A(G)^2)=2e(G),\qquad\tr(A(G)^3)=6t(G).\]
\end{lemma}

Fix an integer $q\ge2$ and put
\[n=q^2+q+1\qquad\text{and}\qquad m_q=\frac12q(q+1)^2.\]
Let $G$ be a connected $C_4$-free graph of order $n$ and write $e(G)=m_q-h$, where $h$ is allowed to be negative.
For each $i\in V(G)$, let $r_i=q+1-d_i$. Next we define the symmetric matrix
\[Q(G):=(q+1)I_n-L(G)=A(G)+\operatorname{diag}(r_{1},r_{2},\ldots,r_{n}).\]
Since $L(G)\mathrm{j}_n=0$, the vector $\mathrm{j}_n$ is an eigenvector of $Q(G)$ with eigenvalue $q+1$. Denote the eigenvalues of $Q(G)$ by $\mu_1,\mu_2,\ldots,\mu_n$, where $\mu_1=q+1$. Since $G$ is connected, $0$ is an eigenvalue of $L(G)$ with multiplicity one, and all the remaining eigenvalues of $L(G)$ are positive. Hence $\mu_i<q+1$ for $2\le i\le n$. By Theorem~\ref{thm:matrix-tree},
\begin{equation}\label{eq:Q-mtt}
n\tau(G)=\prod_{i=2}^n(q+1-\mu_i).
\end{equation}

The $C_4$-free condition provides the following useful relations for the first three spectral moments.   Unless otherwise stated, all subsequent results in this section use the notation introduced above and assume that $G$ is a connected $C_4$-free graph of order $n=q^2+q+1$.

\begin{lemma}\label{lem:moments}
We keep the notation and assumptions introduced above. Define
\[
R_1=\sum_{i=1}^n r_i,\qquad R_2=\sum_{i=1}^n r_i(r_i-1),\qquad R_3=\sum_{i=1}^n r_i(r_i-1)^2.
\]
Then $R_1=q+1+2h$ and
\begin{align}\label{eq:mom1}
	\sum_{i=2}^n\mu_i=2h,
	\qquad
	\sum_{i=2}^n(\mu_i^2-q)
	=R_2,
	\qquad
	\sum_{i=2}^n(\mu_i^3-q\mu_i)
	\ge R_3.
\end{align}
\end{lemma}

\begin{proof}
Since
\[
\sum_{i=1}^n d_i=2e(G)=2m_q-2h,
\]
we have
\begin{equation*}
R_1=\sum_{i=1}^n(q+1-d_i)=n(q+1)-2e(G)
=n(q+1)-2m_q+2h=q+1+2h.
\end{equation*}
Since $\tr(Q(G))=R_1$ and $\mu_1=q+1$, we obtain
\[
\sum_{i=2}^n\mu_i=\tr(Q(G))-(q+1)=R_1-(q+1)=2h.
\]

Let $D_r=\operatorname{diag}(r_{1},r_{2},\ldots,r_{n})$ be a diagonal matrix. Since $Q(G)=A(G)+D_r$,
\[
\tr(Q(G)^2)=\tr(A(G)^2)+\tr(A(G)D_r)+\tr(D_rA(G))+\tr(D_r^2).
\]
Since the diagonal entries of $A(G)$ are zero, $\tr(A(G)D_r)=\tr(D_rA(G))=0$. By Lemma~\ref{lem:adjacency-walks},
\begin{equation*}
\begin{split}
\tr(Q(G)^2)=\tr(A(G)^2)+\sum_{i=1}^n r_i^2=2e(G)+\sum_{i=1}^n r_i^2
=2m_q-2h+\sum_{i=1}^n r_i^2.
\end{split}
\end{equation*}
Hence
\begin{align*}
\sum_{i=2}^n(\mu_i^2-q)
&=\tr(Q(G)^2)-(q+1)^2-q(n-1)\\
&=2m_q-2h+\sum_{i=1}^n r_i^2-(q+1)^2-q(n-1)\\
&=\sum_{i=1}^n r_i^2-(q+1)-2h\\
&=\sum_{i=1}^n r_i^2-R_1=\sum_{i=1}^n(r_i^2-r_i)=R_2.
\end{align*}

Observe that $\tr(A(G)D_r^2)=0$ and $\tr(A(G)^2D_r)=\sum_{i=1}^n d_ir_i$. For the cubic moment, by Lemma~\ref{lem:adjacency-walks},
\begin{equation*}
\begin{split}
\tr(Q(G)^3)&=\tr(A(G)^3)+3\tr(A(G)^2D_r)+3\tr(A(G)D_r^2)+\tr(D_r^3)\\
&=6t(G)+3\sum_{i=1}^n d_ir_i+\sum_{i=1}^n r_i^3.
\end{split}
\end{equation*}
Recall that  $n=q^2+q+1$, $d_i=q+1-r_i$ and $R_1=q+1+2h$. Using \eqref{eq:triangle-lower}, we obtain
\begin{align*}
\tr(Q(G)^3)&\ge\sum_{i=1}^n d_i^2-n(n-1)+3\sum_{i=1}^n d_ir_i+\sum_{i=1}^n r_i^3\\
&=\sum_{i=1}^n (q+1-r_i)(q+1)
-n(n-1)+2\sum_{i=1}^n (q+1-r_i)r_i+\sum_{i=1}^n r_i^3\\
&=n(q+1)^2-n(n-1)+(q+1)R_1-2\sum_{i=1}^n r_i^2+\sum_{i=1}^n r_i^3\\
&=(q+1)^{3}+2qh+q+1+2h-2\sum_{i=1}^n r_i^2+\sum_{i=1}^n r_i^3.
\end{align*}
Since $\mu_1=q+1$, we have
\[
\sum_{i=2}^n\mu_i^3=\tr(Q(G)^3)-(q+1)^3.
\]
By \eqref{eq:mom1}, $q\sum_{i=2}^n\mu_i=2qh$. Therefore,
\begin{align*}
\sum_{i=2}^n(\mu_i^3-q\mu_i)&=\tr(Q(G)^3)-(q+1)^3-2qh\\
&\ge q+1+2h-2\sum_{i=1}^n r_i^2+\sum_{i=1}^n r_i^3\\
&= \sum_{i=1}^n r_i-2\sum_{i=1}^n r_i^2+\sum_{i=1}^n r_i^3\\
&=\sum_{i=1}^n r_i(r_i-1)^2=R_3.
\end{align*}
Thus, \eqref{eq:mom1} holds. This completes the proof.
\end{proof}

We now recall some standard facts about Hermite interpolation. Let $x_0,x_1,\ldots,x_s$ be distinct real numbers and let
$\nu_0,\nu_1,\ldots,\nu_s$ be positive integers. Set $N=\sum_{i=0}^s\nu_i$. For an $N$ times differentiable function $f(x)$, the {\it Hermite interpolating polynomial} corresponding to the nodes $x_0,x_1,\ldots,x_s$ with multiplicities $\nu_0,\nu_1,\ldots,\nu_s$ is the unique polynomial $P(x)$ of degree at most $N-1$ satisfying
\[P^{(k)}(x_i)=f^{(k)}(x_i),\qquad0\le i\le s,\quad 0\le k\le\nu_i-1.\]
Thus, the multiplicity $\nu_i$ means that $P(x)$ and $f(x)$ agree at $x_i$ up to derivatives of order $\nu_i-1$. In particular, if $\nu_i=1$ for every $i$, this reduces to ordinary polynomial interpolation. We shall use the following standard remainder formula; see \cite[Section~2.1.5]{StoerBulirsch2002}.

\begin{lemma}[\hspace{1sp}{\cite{StoerBulirsch2002}}]\label{lem:Hermite-remainder}
Let $x_0,x_1,\ldots,x_s$ be distinct real numbers, let $\nu_0,\nu_1,\ldots,\nu_s$ be positive integers, and set
$N=\sum_{i=0}^s\nu_i$. Suppose that $f(x)$ is $N$ times differentiable on an interval containing $x_0,x_1,\ldots,x_s,x$, and let $P(x)$ be the Hermite interpolating polynomial of degree at most $N-1$ satisfying
\[P^{(k)}(x_i)=f^{(k)}(x_i),\qquad0\le i\le s,\quad 0\le k\le\nu_i-1.\]
Then there exists a point $\xi$ in the smallest interval containing $x_0,x_1,\ldots,x_s,x$ such that
\begin{equation*}
f(x)-P(x)=\frac{f^{(N)}(\xi)}{N!}\prod_{i=0}^s(x-x_i)^{\nu_i}.
\end{equation*}
\end{lemma}

\section{Polarity graphs}

A projective plane of order $q$ has $q^2+q+1$ points and $q^2+q+1$ lines. Each line contains exactly $q+1$ points, and each point lies on exactly $q+1$ lines. Moreover, any two distinct points lie on a unique line, and any two distinct lines meet in a unique point. We record three facts about polarity graphs that will be used in the proof of the main result.

First, by the Matrix-Tree Theorem, London~\cite{London2026} obtained the following exact formula for the number of spanning trees of a polarity graph.
\begin{lemma}[\hspace{1sp}{\cite{London2026}}]
\label{lem:London-polarity}
Let $G$ be a polarity graph arising from a projective plane of order $q$, and let $n=q^2+q+1$. Then
\[
\tau(G)=n^{(n-3)/2}\left(\frac{q+1-\sqrt q}{q+1+\sqrt q}\right)^{(m_q-e(G))/\sqrt q}.
\]
\end{lemma}

The next lemma concerns the vertex degrees of polarity graphs; see also \cite[Lemma~2.3]{HMY2021}.

\begin{lemma}
	\label{pola-de}
Let $G=G(\varphi)$ be a polarity graph arising from a projective plane of order $q$. Then, for every $i\in V(G)$,
\[
d_i=
\begin{cases}
q, & \text{if $i$ is an absolute point},\\
q+1, & \text{otherwise}.
\end{cases}
\]
Consequently, if $r_i=q+1-d_i$, then $r_i=1$ precisely when $i$ is an absolute point and $r_i=0$ otherwise. In particular, $\sum_{i\in V(G)}r_i$ is the number of absolute points of $\varphi$.
\end{lemma}

\begin{proof}
Each polar line contains exactly \(q+1\) points. Hence a vertex has degree \(q\) if it is an absolute point, and degree \(q+1\) otherwise. Thus Lemma~\ref{pola-de} follows directly.
\end{proof}

For an incidence structure with points $p_1,p_2,\ldots,p_a$ and lines $\ell_1,\ell_2,\ldots,\ell_b$, its incidence matrix is the $0$--$1$ matrix $M=(M_{ij})_{b\times a}$
defined by
\[M_{ij}=1\quad\text{if and only if}\quad p_j\in\ell_i.\]
A projective plane is an incidence structure in which any two distinct points lie on a unique line, any two distinct lines meet in a unique point, and there exist four points no three of which are collinear. It has order $q$ if every line contains exactly $q+1$ points. In this case, every point lies on exactly $q+1$ lines, and there are $q^2+q+1$ points and $q^2+q+1$ lines. Hence the incidence matrix of a projective plane of order $q$ is square. Finally, we record a matrix characterization that will be useful in the equality case.

\begin{lemma}\label{lem:matrix-polarity}
Let $q\ge2$ and $n=q^2+q+1$. Suppose that $M$ is a symmetric $0$--$1$ matrix of order $n$ satisfying
\[M^2=qI_n+J_n.\]
Then $M$ is the incidence matrix of a projective plane of order $q$, and it induces a polarity of this plane. Consequently, the matrix obtained from $M$ by replacing its diagonal entries by zero is the adjacency matrix of the corresponding polarity graph.
\end{lemma}

\begin{proof}
Take $\{1,2,\ldots,n\}$ as the point set, and for each $1\le i\le n$, define
\[\ell_i=\{j\in\{1,2,\ldots,n\}: M_{ij}=1\}.\]
We regard $\ell_1,\ell_2,\ldots,\ell_n$ as the lines. Since $M$ is symmetric and has entries in $\{0,1\}$,
\[|\ell_i|=\sum_{k=1}^n M_{ik}^2=(M^2)_{ii}=q+1.\]
Thus every $\ell_i$ contains exactly $q+1$ points. For distinct $i,j$,
\begin{equation}\label{hhh-1.1}
|\ell_i\cap\ell_j|=\sum_{k=1}^n M_{ik}M_{jk}=(MM^{\top})_{ij}=(M^2)_{ij}=1.
\end{equation}
Hence every two distinct lines meet in exactly one point.

Similarly, since $M^{\top}M=M^2$, for distinct points $i,j$, the entry $(M^{\top}M)_{ij}=1$ shows that there is exactly one line containing both $i$ and $j$. Thus every two distinct points lie on a unique line. The lines $\ell_1,\ell_2,\ldots,\ell_n$ are also distinct, since two equal rows would have inner product $q+1$, whereas two distinct rows have inner product $1$ by \eqref{hhh-1.1}.

It remains to show that there exist four points no three of which are collinear. Choose distinct points $x,y$ on a line and a point $z$ outside that line. Note that such a point $z$ exists because $n=q^2+q+1>q+1$. Then $x,y,z$ are noncollinear. Let $\ell_{xy}$, $\ell_{xz}$, and $\ell_{yz}$ denote the unique lines through the pairs $\{x,y\}$,
$\{x,z\}$, and $\{y,z\}$, respectively. These three lines are distinct and together contain $3(q+1)-3=3q$ points. Since $n-3q=(q-1)^2>0$, there is a point outside their union. Hence there exist four points no three of which are collinear. Therefore the incidence structure formed by the points $1,2,\ldots,n$ and the lines $\ell_1,\ell_2,\ldots,\ell_n$ is a projective plane of order $q$.

Now define
\[\varphi(i)=\ell_i,\qquad\varphi(\ell_i)=i,\qquad 1\le i\le n.\]
Since the lines $\ell_1,\ell_2,\ldots,\ell_n$ are distinct, $\varphi$ is a bijection between the point set and the line set, and $\varphi^2$ is the identity.

It remains to verify the incidence condition for a polarity. Let $i$ be a point and let $\ell_j$ be a line. If $i\in\ell_j$, then $M_{ji}=1$. Since $M$ is symmetric, $M_{ij}=1$, and hence $j\in\ell_i$. The converse follows in the same way. Therefore
\[
i\in\ell_j\quad\Longleftrightarrow\quad\varphi(\ell_j)\in\varphi(i).
\]
Thus $\varphi$ is a polarity of the projective plane.

Finally, for distinct points $i,j$, the symmetry of $M$ gives
\[M_{ij}=1\quad\Longleftrightarrow\quad i\in\ell_j=\varphi(j).
\]
Hence, after replacing the diagonal entries of $M$ by zero, the resulting matrix is precisely the adjacency matrix of the polarity graph associated with $\varphi$.
\end{proof}

\section{Some auxiliary functions and inequalities}
Throughout this section, let $q\ge2$ be an integer and $n=q^2+q+1$. Recall $A_q$ from~\eqref{intro-cons}, and define
\begin{equation}\label{eq:AandB}
B_q=\frac{(q+1)/n-A_q}{2q}.
\end{equation}

\begin{lemma}\label{se4l1-e1}
For any integer $q\ge 2$,
\begin{equation}
	\label{eq:uniform-coefficients}
A_q>\frac{1}{q+1}+\frac{q}{3(q+1)^3},
\end{equation}
\begin{equation}\label{eq:B-upper}
0<B_q<\frac{2(q+1)^2+q}{6n(q+1)^3},
\end{equation}
and
\begin{equation}\label{eq:B-small}
B_q(q+2)<\frac{1}{2n}.
\end{equation}
\end{lemma}

\begin{proof}
By the definition of $A_q$ and direct integration,
\[A_q=(q+1)\int_0^1\frac{dt}{(q+1)^2-qt^2}=\frac{1}{q+1}\int_0^1\frac{dt}{1-\frac{q}{(q+1)^2}t^2}.\]
Since $\frac{1}{1-x}\geq1+x$ for $0\leq x<1$, with strict inequality for $x>0$, we have
\[
A_q>\frac{1}{q+1}\int_0^1\left(1+\frac{q}{(q+1)^2}t^2\right)dt=\frac{1}{q+1}+\frac{q}{3(q+1)^3},
\]
which proves \eqref{eq:uniform-coefficients}.

Moreover, since $0\le t^2\le1$,
\[\frac{1}{1-\frac{q}{(q+1)^2}t^2}\le\frac{1}{1-\frac{q}{(q+1)^2}},\]
with strict inequality for $0\le t<1$. Hence
\[
A_q<\frac{1}{q+1}\int_0^1\frac{dt}{1-\frac{q}{(q+1)^2}}=\frac{q+1}{n},
\]
where we have used $n=(q+1)^2-q$. Hence $B_q>0$. Combining this with \eqref{eq:uniform-coefficients},
\[
B_q<\frac{1}{2q}\left(\frac{q+1}{n}-\frac{1}{q+1}-\frac{q}{3(q+1)^3}\right)=\frac{2(q+1)^2+q}{6n(q+1)^3},
\]
which proves \eqref{eq:B-upper}.

Finally, since $q\ge2$,
\[
\frac{1}{2n}-\frac{(q+2)(2(q+1)^2+q)}{6n(q+1)^3}=\frac{q^3-3q-1}{6n(q+1)^3}>0.
\]
Together with \eqref{eq:B-upper}, this gives
\[
B_q(q+2)<\frac{1}{2n},
\]
and hence \eqref{eq:B-small} follows.
\end{proof}

\begin{lemma}
	\label{lem:uniform-tail}
Let $q\ge2$ be an integer. For every $x\in\mathbb R$ with $x<q+1$, let
\[
P(x)=\frac12\ln n-A_qx-\frac{x^2-q}{2n}-B_q(x^3-qx),
\]
and let
\[
g(x)=P(x)-\ln(q+1-x).
\]
Then $g(x)\ge0$, with equality if and only if $x\in\{-\sqrt q,\sqrt q\}$. Moreover, the function
\[
\phi(x)=
\begin{cases}
g(x),&\text{if }x\le-\sqrt q,\\
0,&\text{if }x>-\sqrt q.
\end{cases}
\]
is convex and non-increasing on $\mathbb R$.
\end{lemma}

\begin{proof}
Let $f(x)=\ln(q+1-x)$ be a function on the interval $(-\infty,q+1)$. Using
\eqref{intro-cons},
 \eqref{eq:AandB} and $n=(q+1-\sqrt q)(q+1+\sqrt q)$, we obtain
\[
P(\pm\sqrt q)=f(\pm\sqrt q),
\qquad
P'(\pm\sqrt q)=f'(\pm\sqrt q).
\]
Thus $P(x)$ is the cubic Hermite interpolating polynomial of $f(x)$ at $-\sqrt q$ and $\sqrt q$. By Lemma  \ref{lem:Hermite-remainder}, for $x\ne\pm\sqrt q$, there exists $\xi$ in the smallest interval $\Lambda$ containing $-\sqrt q$, $\sqrt q$, and $x$ such that
\begin{equation}\label{xxx1}
f(x)-P(x)=\frac{f^{(4)}(\xi)}{4!}(x^2-q)^2.
\end{equation}
Since $x<q+1$ and $\sqrt q<q+1$, the above smallest interval $\Lambda$  is contained in $(-\infty,q+1)$. Moreover, for $y<q+1$,
\[f^{(4)}(y)=-\frac{6}{(q+1-y)^4}<0.\]
By \eqref{xxx1}, we have $f(x)<P(x)$ unless $x=\pm\sqrt q$. Equivalently, $g(x)\ge0$ for all $x<q+1$, with equality if and only if $x=\pm\sqrt q$.

It remains to prove that $\phi(x)$ is convex and non-increasing on $\mathbb R$.  Since $g(x)\ge0$ for all $x<q+1$ and $g(-\sqrt q)=0$, the point $-\sqrt q$ is a local minimum point of $g(x)$. Therefore,
\begin{equation}\label{xxx-1.1}
g''(-\sqrt q)\ge0.
\end{equation}

We next show that $g'''(x)$ has a zero in $(-\sqrt q,\sqrt q)$. Since $P'(\pm\sqrt q)=f'(\pm\sqrt q)$, we have $g'(\pm\sqrt q)=0$. Moreover,  since $g(-\sqrt q)=g(\sqrt q)=0$, Rolle's theorem gives a point $\xi\in(-\sqrt q,\sqrt q)$ such that $g'(\xi)=0$. Thus
\[
g'(-\sqrt q)=g'(\xi)=g'(\sqrt q)=0.
\]
Applying Rolle's theorem to $g'(x)$ on the intervals $[-\sqrt q,\xi]$ and $[\xi,\sqrt q]$, respectively, we obtain points $\xi_1\in(-\sqrt q,\xi)$ and $\xi_2\in(\xi,\sqrt q)$ such that $g''(\xi_1)=g''(\xi_2)=0$. Applying Rolle's theorem once more, there exists $\eta\in(\xi_1,\xi_2)\subset(-\sqrt q,\sqrt q)$ such that $g'''(\eta)=0$.

A routine computation gives
\[g'''(x)=-6B_q+\frac{2}{(q+1-x)^3},\qquad g^{(4)}(x)=\frac{6}{(q+1-x)^4}>0.
\]
Hence $g'''(x)$ is strictly increasing on $(-\infty,q+1)$. Now let $x\le-\sqrt q$. Since $\eta\in(-\sqrt q,\sqrt q)$, we have $x<\eta$, and therefore $g'''(x)<g'''(\eta)=0$. Thus $g''(x)$ is  strictly decreasing on $(-\infty,-\sqrt q]$, and by \eqref{xxx-1.1}, we have $g''(x)\ge g''(-\sqrt q)\ge0$. It follows that $g'(x)$ is strictly increasing on $(-\infty,-\sqrt q]$.  Since $g'(-\sqrt q)=0$, we have $g'(x)\le0$. Hence $g(x)$ is convex and non-increasing on $(-\infty,-\sqrt q]$. Since $\phi(x)=0$ for $x>-\sqrt q$ and
\[
g(-\sqrt q)=g'(-\sqrt q)=0,
\]
the function $\phi$ is convex and non-increasing on $\mathbb R$.
\end{proof}

\begin{lemma}\label{lem:uniform-scalar}
Let $q\ge2$ be an integer, and let $\phi(x)$ be the function defined in Lemma~\ref{lem:uniform-tail}. Then, for every integer $r\ge1-q^2$,
\begin{equation}\label{eq:uniform-scalar}
r(r-1)\left(B_q(1-r)-\frac{1}{2n}\right)\le\phi(r),
\end{equation}
with equality if and only if $r\in\{0,1\}$.
\end{lemma}

\begin{proof}
For a nonnegative integer $r$, the assertion is immediate. Indeed, the left-hand side of \eqref{eq:uniform-scalar} is zero for $r=0,1$ and is strictly negative for $r\ge2$ since $B_q>0$, whereas $\phi(r)=0$. Thus equality holds for $r=0,1$ and is strict for $r\ge2$.

Now suppose that $r<0$. Since $r\ge1-q^2$, we have $1-q^2\le r\le-1$. If $-\sqrt q\le r\leq-1$, then $\phi(r)=0$, since $\phi(x)=0$ for $x>-\sqrt q$ and $g(-\sqrt q)=0$. Moreover,
\[
1-r\le\sqrt q+1<q+2.
\]
Since $B_q>0$ and \eqref{eq:B-small} gives
\[
B_q(q+2)<\frac{1}{2n},
\]
we obtain
\[r(r-1)\left(B_q(1-r)-\frac{1}{2n}\right)<r(r-1)\left(B_q(q+2)-\frac{1}{2n}\right)<0.
\]
Hence \eqref{eq:uniform-scalar} is strict in this case.

It remains to consider $1-q^2\le r<-\sqrt q$. For convenience, we prove the stronger statement for all real $r\in[1-q^2,-\sqrt q]$. Since $r\le-\sqrt q$, we have $\phi(r)=g(r)$. A direct computation gives
\begin{align}
&r(r-1)\left(B_q(1-r)-\frac{1}{2n}\right)-\phi(r)\notag\\
&\qquad=\Phi(r):=\ln(q+1-r)-\frac12\ln n-\frac{q}{2n}+2B_qr^2+\left(A_q+\frac{1}{2n}-B_q(q+1)\right)r.
\label{eq:uniform-E}
\end{align}
We show that $\Phi(r)<0$ for all real $r\in[1-q^2,-\sqrt q]$.

Differentiating gives
\[
\Phi'(r)=-\frac{1}{q+1-r}+4B_qr+A_q+\frac{1}{2n}-B_q(q+1),
\qquad
\Phi'''(r)=-\frac{2}{(q+1-r)^3}<0.
\]
Thus $\Phi''(r)$ is strictly decreasing on $[1-q^2,-\sqrt q]$, and $\Phi'(r)$ is a strictly concave function  on this interval. Hence $\Phi'(r)$ is either monotone or first increases and then decreases. Using
\eqref{eq:AandB} together with $(q+1+\sqrt q)(q+1-\sqrt q)=n$, we obtain
\begin{align*}
2n\Phi'(-\sqrt q)
&=(2\sqrt q+1)-2nB_q(3q+4\sqrt q+1)\\
&>(2\sqrt q+1)-\frac{3q+4\sqrt q+1}{q+2}\\
&=\frac{2q(\sqrt q-1)+1}{q+2}>0,
\end{align*}
where the first inequality follows from \eqref{eq:B-small}.

Put $a=1-q^2$ and $b=-\sqrt q$. Since $\Phi'$ is concave on $[a,b]$ and $\Phi'(b)>0$, once $\Phi'$ becomes nonnegative, it remains positive to the right. Indeed, if $a\le u<v\le b$ and $\Phi'(u)\ge0$, then concavity gives
\[
\Phi'(v)\ge
\frac{b-v}{b-u}\,\Phi'(u)
+\frac{v-u}{b-u}\,\Phi'(b)>0.
\]
Consequently, $\Phi'$ is either nonnegative throughout $[a,b]$, or changes sign exactly once, from negative to positive. Thus $\Phi$ either increases throughout the interval or first decreases and then increases. In either case,
\[
\max_{r\in[a,b]}\Phi(r)
=\max\{\Phi(a),\Phi(b)\}.
\]
It therefore suffices to show that both endpoint values are negative.

At the endpoint $r=-\sqrt q$, since $\phi(-\sqrt q)=0$, \eqref{eq:B-small} yields
\begin{equation*}
\Phi(-\sqrt q)
=\sqrt q(\sqrt q+1)\left(B_q(\sqrt q+1)-\frac{1}{2n}\right)
< \sqrt q(\sqrt q+1)\left(B_q(q+2)-\frac{1}{2n}\right)<0.
\end{equation*}

For the other endpoint $r=1-q^2$, set
\[
K_0=(q+1)\left[\frac{1}{q+1}+\frac{q}{3(q+1)^3}+\frac{1}{2n}-\frac{(2(q+1)^2+q)(q+1)(2q-1)}{6n(q+1)^3}\right].
\]
By \eqref{eq:uniform-coefficients} and \eqref{eq:B-upper}, we have
\begin{equation}\label{xxx-1.2}
(q+1)\left(A_q+\frac{1}{2n}-B_q(q+1)(2q-1)\right)>K_0.
\end{equation}
Since $q^2-1=(q-1)(q+1)$ and $\frac{(q+1)^2}{n}=1+\frac{q}{n}$, substituting $r=1-q^2$ into \eqref{eq:uniform-E} gives
\begin{align*}
\Phi(1-q^2)
=\ln q+\frac12\ln\left(1+\frac{q}{n}\right)-\frac{q}{2n}-(q-1)(q+1)\left(A_q+\frac{1}{2n}-B_q(q+1)(2q-1)\right).
\end{align*}
Hence, by \eqref{xxx-1.2},
\[
\Phi(1-q^2)<\ln q+\frac12\ln\left(1+\frac{q}{n}\right)-\frac{q}{2n}-(q-1)K_0.
\]
Using $\ln(1+x)<x$ for $x>0$, we obtain
\begin{equation}\label{eq:E-endpoint}
\Phi(1-q^2)<\ln q-(q-1)K_0.
\end{equation}
A direct simplification gives
\begin{equation}\label{xxx-1.3}
K_0-\frac{q+4}{4q}=\frac{q^5+q^4+8q^3+7q^2-17q-12}{12q(q+1)^2n}>0,
\end{equation}
where the inequality holds since $q^5+q^4+8q^3+7q^2-17q-12=(q-2)(q^4+3q^3+14q^2+35q+53)+94$. Finally, for $x\ge1$, let
\[
F(x)=\frac{(x-1)(x+4)}{4x}-\ln x.
\]
We have $F(1)=0$ and
\[
F'(x)=\frac{(x-2)^2}{4x^2}\ge0\qquad(x\ge1),
\]
with equality only at $x=2$. Hence $F$ is strictly increasing on $[1,\infty)$, and in particular $F(q)>0$ for every $q\ge2$.
Hence, since $q\ge2$ and by \eqref{xxx-1.3}, we have
\[
\ln q<\frac{(q-1)(q+4)}{4q}<(q-1)K_0.
\]
Together with \eqref{eq:E-endpoint}, this yields $\Phi(1-q^2)<0$.

Both endpoint values are negative, and hence $\Phi(r)<0$ for $1-q^2\le r\le-\sqrt q$. Thus \eqref{eq:uniform-scalar} is strict in the remaining case. Combining this with the preceding cases, equality holds if and only if $r\in\{0,1\}$.
\end{proof}

\section{Proof of the main results}

We first prove Theorem~\ref{thm:uniform-bound} and then deduce Corollary~\ref{cor:edge-bound}. Throughout, let $n=q^2+q+1$, $m_q=\frac12q(q+1)^2$, and $h=m_q-e(G)$.

\begin{proof}[Proof of Theorem~\ref{thm:uniform-bound}]
Recall from Section~2 that $Q(G)$ is symmetric. Hence there exists an orthogonal matrix
$U=(u_1,u_2,\ldots,u_n)$ such that
\begin{equation}\label{hhhh-1.2}
Q(G)=U\operatorname{diag}(\mu_1,\mu_2,\ldots,\mu_n)U^{\top}.
\end{equation}
Since $Q(G)\mathrm{j}_n=(q+1)\mathrm{j}_n$, we choose $\mu_1=q+1$ and $u_1=\frac{1}{\sqrt n}\mathrm{j}_n$. For each $1\le i\le n$,
\[
r_i=Q(G)_{ii}=\sum_{k=1}^n U_{ik}^2\mu_k,
\qquad
\sum_{k=1}^n U_{ik}^2=1.
\]
By Lemma~\ref{lem:uniform-tail}, the function $\phi$ is convex. Hence Jensen's inequality gives
\[
\phi(r_i)\le\sum_{k=1}^n U_{ik}^2\phi(\mu_k).
\]
Summing over $i=1,2,\ldots,n$ and using
$\sum_{i=1}^n U_{ik}^2=1$, we obtain
\[
\begin{aligned}
\sum_{i=1}^n\phi(r_i)\le\sum_{i=1}^n\sum_{k=1}^n U_{ik}^2\phi(\mu_k)=\sum_{k=1}^n\phi(\mu_k)\sum_{i=1}^nU_{ik}^2=\sum_{k=1}^n\phi(\mu_k).
\end{aligned}
\]
Since $\mu_1=q+1$ and $\phi(q+1)=0$, it follows that
\[
\sum_{i=1}^n\phi(r_i)\le\sum_{k=2}^n\phi(\mu_k).
\]
Recall from Lemma~\ref{lem:uniform-tail} that
$g(x)=P(x)-\ln(q+1-x)$. Since $\mu_k<q+1$ for
$2\le k\le n$ and $\phi(x)\le g(x)$ for every $x<q+1$, we obtain
\begin{equation}\label{eq:uniform-Jensen}
\sum_{k=2}^n g(\mu_k)\ge\sum_{k=2}^n\phi(\mu_k)\ge\sum_{i=1}^n\phi(r_i).
\end{equation}

Observe that $\ln(q+1-\mu_k)=P(\mu_k)-g(\mu_k)$. Using \eqref{eq:Q-mtt}, we obtain
\begin{equation}\label{xxx-2.2}
\begin{split}
\ln\frac{\tau(G)}{n^{(n-3)/2}}&=\sum_{k=2}^{n}\ln(q+1-\mu_k)-\ln n-\ln n^{(n-3)/2}\\
&=\sum_{k=2}^{n}\left(P(\mu_k)-g(\mu_k)\right)-\frac{n-1}{2}\ln n\\
&=-A_q\sum_{k=2}^n\mu_k-\frac{1}{2n}\sum_{k=2}^n(\mu_k^2-q)-B_q\sum_{k=2}^n(\mu_k^3-q\mu_k)-\sum_{k=2}^n g(\mu_k).
\end{split}
\end{equation}
Recall that $B_q>0$ in \eqref{eq:B-upper}. By Lemma~\ref{lem:moments} and \eqref{eq:uniform-Jensen}, we have
\begin{equation}\label{xxx-2.1}
\begin{split}
\ln\frac{\tau(G)}{n^{(n-3)/2}}
&\le-2A_qh-\frac{R_2}{2n}-B_qR_3-\sum_{i=1}^n\phi(r_i)\\
&=-2A_qh+\sum_{i=1}^n\left[r_i(r_i-1)\left(B_q(1-r_i)-\frac{1}{2n}\right)-\phi(r_i)
\right].
\end{split}
\end{equation}
Since $G$ is simple and connected, $d_i\le n-1=q^2+q$, and hence $r_i=q+1-d_i\ge1-q^2$ for each $1\le i\le n$. Therefore, by Lemma~\ref{lem:uniform-scalar} and \eqref{xxx-2.1},
\[
\ln\frac{\tau(G)}{n^{(n-3)/2}}
\le -2A_qh.
\]
Exponentiating gives \eqref{eq:intro-uniform-bound}.

We next determine the equality case. Suppose that equality holds in \eqref{eq:intro-uniform-bound}. Since every summand in \eqref{xxx-2.1} is nonpositive, equality in Lemma~\ref{lem:uniform-scalar} must hold for every $1\le i\le n$. Hence $r_i\in\{0,1\}$ for each  $1\le i\le n$. Consequently,
\[R_2=R_3=0\qquad\text{and}\qquad\phi(r_i)=0,\qquad 1\le i\le n.\]

Returning to \eqref{xxx-2.2} and using $\sum\limits_{k=2}^n\mu_k=2h$ and $\sum\limits_{k=2}^n(\mu_k^2-q)=R_2=0$, equality in \eqref{eq:intro-uniform-bound} gives
\[B_q\sum_{k=2}^n(\mu_k^3-q\mu_k)+\sum_{k=2}^n g(\mu_k)=0.\]
By Lemma~\ref{lem:moments},
\[
\sum_{k=2}^n(\mu_k^3-q\mu_k)\ge R_3=0,
\]
while $g(\mu_k)\ge0$ for every $2\le k\le n$. Since $B_q>0$, it follows that $\sum_{k=2}^n g(\mu_k)=0$. Therefore
$g(\mu_k)=0$ for each $2\le k\le n$. By Lemma~\ref{lem:uniform-tail}, $g(x)=0$ if and only if $x\in\{-\sqrt q,\sqrt q\}$. Then $\mu_k\in\{-\sqrt q,\sqrt q\}$ for each $2\le k\le n$. Moreover, the matrix $Q(G)^2$ has eigenvalue $(q+1)^2$ with multiplicity $1$ and eigenvalue $\mu_k^2=q$ with multiplicity $n-1$. Hence, using the orthogonal diagonalization in \eqref{hhhh-1.2},
\[
\begin{aligned}
Q(G)^2&=U\operatorname{diag}\bigl((q+1)^2,q,\ldots,q\bigr)U^{\top}\\
&=U\left[qI_n+\bigl((q+1)^2-q\bigr)\operatorname{diag}(1,0,\ldots,0)\right]U^{\top}\\
&=qI_n+\bigl((q+1)^2-q\bigr)u_1u_1^{\top}\\
&=qI_n+\frac{(q+1)^2-q}{n}J_n\\
&=qI_n+J_n,
\end{aligned}
\]
where $u_1=\frac{1}{\sqrt n}\mathrm{j}_n$ and $(q+1)^2-q=n$.

Since $r_i\in\{0,1\}$ for every $1\le i\le n$, the matrix
\[
Q(G)=A(G)+\operatorname{diag}(r_1,r_2,\ldots,r_n)
\]
is a symmetric $0$--$1$ matrix. Applying Lemma~\ref{lem:matrix-polarity} with $M=Q(G)$, and noting that replacing the diagonal entries of $Q(G)$ by zero yields the adjacency matrix $A(G)$, we conclude that $G$ is a polarity graph arising from a projective plane of order $q$.

Conversely, suppose that $G$ is a polarity graph arising from a projective plane of order $q$. By Lemma \ref{lem:London-polarity} and $h=m_q-e(G)$,
\[\tau(G)=n^{(n-3)/2}\left(\frac{q+1-\sqrt q}{q+1+\sqrt q}\right)^{h/\sqrt q}=n^{(n-3)/2}\exp(-2A_qh).\]
Thus equality holds in \eqref{eq:intro-uniform-bound}. This completes the proof.
\end{proof}

We can now deduce Corollary~\ref{cor:edge-bound} immediately.

\begin{proof}[Proof of Corollary~\ref{cor:edge-bound}]
Note that $e(G)\leq \frac{1}{2}q(q+1)^2$, that is, $h=m_q-e(G)\ge0$. By Theorem~\ref{thm:uniform-bound},
\[\tau(G)\le n^{(n-3)/2}\exp(-2A_qh)\le n^{(n-3)/2},\]
since $A_q>0$.

Suppose that equality holds. Then $h=0$, and equality must also hold in Theorem~\ref{thm:uniform-bound}. Hence $G$ is a polarity graph of a projective plane of order $q$. By Lemma~\ref{lem:moments},
\[\sum_{i=1}^n r_i=q+1+2h=q+1.\]
By Lemma~\ref{pola-de}, this is exactly the number of absolute points of the corresponding polarity. Hence $G$ is an orthogonal polarity graph.

Conversely, suppose that $G$ is an orthogonal polarity graph. By Lemma~\ref{pola-de}, $\sum_{i=1}^n r_i=q+1$.
On the other hand, Lemma~\ref{lem:moments} gives $\sum_{i=1}^n r_i=q+1+2h$. Thus $h=0$. Since $G$ is a polarity graph, the equality statement in Theorem~\ref{thm:uniform-bound} yields
\[\tau(G)=n^{(n-3)/2}.\]
Therefore equality holds if and only if $G$ is an orthogonal polarity graph.
\end{proof}

Corollary~\ref{cor:nonexceptional-orders} follows immediately from Corollary~\ref{cor:edge-bound} and Theorem~\ref{th-1}.
For prime powers $q$, the existence of orthogonal polarity graphs $ER_q$ then yields Corollary \ref{cor:london-conjecture}.

\section{Concluding remarks}
In this paper, we studied a Tur\'an-type extremal problem for the number of spanning trees in connected $C_4$-free graphs. Let $n=q^2+q+1$ and $m_q=\frac12q(q+1)^2$. We proved the uniform estimate
\[\tau(G)\le n^{(n-3)/2}\exp\bigl(-2A_q(m_q-e(G))\bigr)\]
for every connected $C_4$-free graph $G$ on $n$ vertices. The proof combines the $C_4$-free codegree condition with spectral moment estimates for $Q(G)=(q+1)I_n-L(G)$, together with Hermite interpolation. The equality case is rigid: equality in the uniform estimate holds precisely for polarity graphs arising from projective planes of order $q$.

As an immediate consequence, every connected $C_4$-free graph $G$ with $e(G)\le m_q$ satisfies
\[\tau(G)\le n^{(n-3)/2},\]
with equality if and only if $G$ is an orthogonal polarity graph. Combining this with F\"uredi's extremal edge bound proves London's conjecture for every prime power $q\notin\{7,9,11,13\}$.

For $q\in\{7,9,11,13\}$, our theorem already proves the desired bound whenever $e(G)\le m_q$; the remaining task is to handle graphs with $e(G)>m_q$. It would therefore be interesting to determine whether the spanning tree bound can be proved for these four cases without first resolving the corresponding extremal edge problem for $C_4$.

A natural direction for further research is to consider the analogous spanning tree extremal problem for general $F$-free graphs. More precisely, given a fixed graph $F$, one may ask which $n$-vertex $F$-free graphs maximize the number of spanning trees, and whether the extremal graphs for the classical Tur\'an problem for $F$ also maximize the number of spanning trees. Graphs excluding longer cycles, such as $C_{2k}$-free and $C_{2k+1}$-free graphs for $k\ge2$, provide natural classes for further investigation.

\section*{Declarations}

\noindent\textbf{Generative AI and AI-assisted technologies.}
During the preparation of this work, the authors used ChatGPT (OpenAI) to assist in exploring potential approaches to the proof of Theorem \ref{thm:uniform-bound}. All AI-generated suggestions were verified and refined by the authors, who take full responsibility for the correctness and originality of the paper.
\smallskip

\noindent\textbf{Funding.}
The work was supported by the National Natural Science Foundation of China (Grant No.\ 12271251).
\smallskip

\noindent\textbf{Conflict of interest.}
The authors declare that they have no conflict of interest.
\smallskip

\noindent\textbf{Data and code availability.}
No datasets were generated or analyzed during this study.

\end{document}